\documentclass{amsart}
\usepackage{verbatim}
\usepackage{rotating} %for \includegraphix[angle=270]
\usepackage{amssymb}
\usepackage{mathrsfs,amsfonts,amsmath,amssymb,epsfig,amscd,xy,amsthm,pb-diagram} 
\newtheorem{theorem}{Theorem}[section]
\newtheorem{proposition}[theorem]{Proposition}

\newtheorem{lemma}[theorem]{Lemma}

\newtheorem{corollary}[theorem]{Corollary}

\newtheorem{definition}[theorem]{Definition}

\theoremstyle{plain}
\theoremstyle{remark}

\newtheorem{remark}[theorem]{Remark}

\newcommand{\C}{{\mathbb C}}

\newcommand{\Q}{{\mathbb Q}}

\newcommand{\Z}{{\mathbb Z}}
\newcommand{\N}{{\mathbb N}}

\newcommand{\cZ}{{\mathcal Z}}

\newcommand{\Gmn}{\mathbb{G}_{\text{m}}^{n}}

\newcommand{\Qbar}{\bar{\Q}}

\DeclareMathOperator{\lcm}{lcm}

\DeclareMathOperator{\Norm}{N}

\newcommand{\bP}{{\mathbb P}}

\newcommand{\bfx}{{\mathbf x}}
\newcommand{\bfe}{{\mathbf e}}
\newcommand{\bfn}{{\mathbf n}}

\newcommand{\bfu}{{\mathbf u}}
\newcommand{\bfr}{{\mathbf r}}

\newcommand{\bA}{{\mathbb A}}

\newcommand{\cO}{\mathcal{O}}
\newcommand{\cD}{\mathcal{D}}

\newcommand{\cP}{\mathcal{P}}

\newcommand{\scrK}{\mathscr{K}}

\newcommand{\bfk}{\mathbf k}

\author{Jason P.~Bell}
\address{
Jason P.~Bell\\
University of Waterloo\\
Department of Pure Mathematics\\
Waterloo, Ontario, Canada N2L 3G1}
\email{jpbell@uwaterloo.ca}

\author{Khoa D.~Nguyen}
\address{
Khoa D.~Nguyen \\
Department of Mathematics and Statistics\\
University of Calgary\\
AB T2N 1N4, Canada
}
\email{dangkhoa.nguyen@ucalgary.ca}

\author{Umberto Zannier}
\address{
Umberto Zannier\\
Scuola Normale Superiore, Classe di Scienze Matematiche e Naturali, Pisa, Italy
}
\email{umberto.zannier@sns.it}


\begin{document}
	\title[D-finiteness]{D-finiteness, rationality, and height}
	\date{May 2019}
	\begin{abstract}		 
		Motivated by a result of 
		van der Poorten and Shparlinski for univariate power series, 
		Bell and Chen prove that if a multivariate power series
		over a field of characteristic $0$ 
		is D-finite and its coefficients belong to a finite set
		then it is a rational function. We extend and strengthen their results 
		to certain power series whose coefficients may form an 
		infinite set. We 
		also prove that if the coefficients of a univariate D-finite
		power series ``look like'' the coefficients of a rational
		function then the power series is rational.
		Our work 
		 relies on the theory of Weil heights, the Manin-Mumford theorem for tori, an application  of the Subspace Theorem, and various combinatorial arguments involving heights, power series, and linear recurrence sequences.
	\end{abstract}
	
	\maketitle
	
	\section{Introduction} \label{sec:intro}
	Let $\N$ denote the set of positive integers and let $\N_0:=\N\cup\{0\}$. 
	Let $m\in \N$ and 
	consider
	the ring $K[[x_1,\ldots,x_m]]$
	of power series in $m$ variables over a field $K$ of 
	characteristic $0$. Very broadly speaking, there are several highly 
	interesting results
	of the following form: if a power series $f$ satisfies the 
	property $\mathcal{P}_1$ and its coefficients satisfy
	the property $\mathcal{P}_2$ which is usually of an arithmetic 
	nature then property
	$\mathcal{P}_3$ holds. 
	For example (when $m=1$), 
	the Pisot's $d$-th Root 
	Conjecture, settled by Zannier \cite{Zan00_AP}, states
	that if $K$ is a number field, $f(x)=\displaystyle\sum_{i\geq 0}a_ix^i \in K[[x]]$ is a 
	rational function, and $a_i$ is
	the $d$-th power of an element of $K$ then
	there exists a rational function
	$g(x)=\displaystyle\sum_{i\geq 0}b_ix^i$ 
	such that $a_i=b_i^d$ for every $i$.
	There is a similar Pisot's Hadamard Quotient 
	Conjecture solved by Pourchet \cite{Pou79_SD} and van der Poorten \cite{vdP88_Sd} (see Rumely's note \cite{Rum88_NO} for more details and the paper by Corvaja-Zannier \cite{CZ02_FO} for a stronger version). 
	Note that $\mathcal{P}_1$ in the above results
	is the property that the given power series is a rational 
	function. 
	In this paper, we are interested in
	the situation when $\mathcal{P}_1$ is the so called 
	\emph{D-finiteness 
	property}.

	Let 
	$\bfn=(n_1,\ldots,n_m)\in \N_0^m$ and let
	$\bfx=(x_1,\ldots,x_m)$ be the vector of the indeterminates $x_1,\ldots,x_m$.
	We write
	$\bfx^{\bfn}$ to denote the monomial
	$x_1^{n_1}\ldots x_m^{n_m}$ having the total degree
	$\Vert \bfn\Vert:=n_1+\ldots+n_m$. 
	We also write
	$\displaystyle\frac{\partial^{\Vert \bfn\Vert}}{\partial\bfx^{\bfn}}$
	to denote the operator
	$$\left(\frac{\partial}{\partial x_1}\right)^{n_1}\ldots
	\left(\frac{\partial}{\partial x_m}\right)^{n_m}$$
	on $K[x_1,\ldots,x_m]$. A power series $f(\bfx)\in K[[\bfx]]$ is said to be 
	\emph{D-finite} (over $K(\bfx)$) if 
	all the derivatives
	$\displaystyle\frac{\partial^{\Vert \bfn\Vert}f}{\partial\bfx^\bfn}$
	for $\bfn\in \N_0^m$
	span a finite-dimensional vector space over $K(\bfx)$.
	Univariate power series satisfying linear differential
	equations 
	(such as the exponential function, hypergeometric series, etc.)
	 have played an important role in mathematics for hundreds 
	 of years. Since the 1960s certain $p$-adic and cohomological
	 aspects of univariate power series solutions of algebraic 
	 differential
	 equations have been developed by Dwork, Katz, and others (see \cite{DGS94_AI} and references therein).

	In 1980, Stanley wrote an expository paper
	\cite{Sta80_DF}
	introducing univariate D-finite power series and many of their properties
	from a combinatorial point of view.
	After that, multivariable D-finite power series were 
	introduced 
	by Lipshitz 
	\cite{Lip89_DF}
	and they have become an important part in enumerative 
	combinatorics especially in the theory of generating functions 
	\cite{Sta99_EC2}. From the linear partial differential equations satisfied
	 by a D-finite series $f(\bfx)$, one can show that the coefficients
	 of $f$ satisfy certain linear recurrence relations with
	 polynomial coefficients. 
	 In particular, if 
	 $f(\bfx)\in\Qbar[[\bfx]]$ is D-finite, the coefficients
	 of $f$ belong to a number field.
	
	Let $h$ denote the absolute
	logarithmic Weil height on $\Qbar$. 
	We have the following
	results of van der Poorten-Shparlinski \cite{vdPS96_OL}
	and Bell-Chen \cite{BC17_PS}:
	\begin{theorem}[van der Poorten-Shparlinski 1996]\label{thm:vdP-S}
	Let $f(x)=\sum_{n\in\N_0}a_nx^n\in \Q[[x]]$
	be a univariate D-finite power series with 
	rational coefficients. If
	$\lim\frac{h(a_n)}{\log\log n}=0$
	then the sequence $(a_n)_{n\in \N_0}$ is periodic.
	\end{theorem}
	
	\begin{theorem}[Bell-Chen 2017]\label{thm:new B-C}
    Let $K$ be a field of characteristic $0$ and 
	let $f(\bfx)=\sum_{\bfn\in\N_0^m}a_{\bfn}\bfx^{\bfn}\in K[[\bfx]]$ be a D-finite power series in $m$ variables. 
	If the coefficients of $f$ belong to a finite set
	then $f$ is rational.
	\end{theorem}
	
	In fact, a slightly more precise version of 
	Theorem~\ref{thm:vdP-S} was proved by 
	van der Poorten-Shparlinski \cite[pp.~147--148]{vdPS96_OL}. 
	Their method uses a technical construction of a
	certain auxiliary function. Although they 
	stated their result
	for power series with rational coefficients, it seems
	that the proof should remain valid over an arbitrary
	number field. 
	
%	The referee also suggests an alternative
%	way to prove the number field version of
%	Theorem~\ref{thm:vdP-S} as follows.
%	Suppose $f(x)\in K[[x]]$ is D-finite where $K$
%	is a (Galois) number field, we can express
%	$f=\sum \alpha_if_i$ where
%	the $\alpha_i$'s are a basis of $K/\Q$
%	and $f_i\in \Q[[x]]$ for every $i$.
%	Then we express each $f_i$ as a linear combination
%	of the Galois conjugates of $f$ to have that
%	each $f_i$ is D-finite and the height of its coefficients
%	has the same asymptotic behavior as the height of
%	the coefficients of $f$. 
	
	After a specialization argument,
	Theorem~\ref{thm:vdP-S} implies that 
	if the coefficients of a univariate 
	D-finite
	power series over 
	a field of characteristic $0$ belong to a finite set then 
	the series is rational. 
	Theorem~\ref{thm:new B-C} is a very recent result of
	Bell-Chen \cite{BC17_PS} generalizing the 
	above consequence
	for multivariate power series. The proof of
	Theorem~\ref{thm:new B-C} in \cite{BC17_PS}
	uses induction on
	the number of variables $m$ and
	various combinatorial arguments 
	involving
	the notion of syndetic subsets of $\N$.
	
	Our first 
	main result
	strengthens and generalizes both 
	Theorem~\ref{thm:new B-C}
	and Theorem~\ref{thm:vdP-S} at one stroke. 
	More specifically, we treat multivariate power series, 
	replace the function
	$\log\log(n)$ in Theorem~\ref{thm:vdP-S} by 
	the more dominant function $\log(n)$, and 
	let one conclude that certain non-rational
	power series are not D-finite even when
	the coefficients do not belong to a finite set.
	We have:
	\begin{theorem}\label{thm:new main 1}
	Let $f(\bfx)=\displaystyle\sum_{\bfn\in\N_0^m}a_{\bfn}\bfx^{\bfn}\in\Qbar[[\bfx]]$.
	Assume that $f$ is D-finite and 
	\begin{equation}\label{eq:thm 1}
	\displaystyle\lim_{\Vert\bfn\Vert\to\infty} \frac{h(a_{\bfn})}{\log\Vert\bfn\Vert}=0.
	\end{equation}
	Then the following hold:
	\begin{itemize}
		\item [(a)] $f$ is a rational function.
		\item [(b)] If $f$ is not a polynomial, its
		denominator, up to scalar multiplication, has the form
		$$\prod_{i=1}^{\ell} (1-\zeta_i\bfx^{\bfn_i})$$
	where $\ell\geq 1$, $\zeta_i$ is a root of unity, 
	$\bfn_i\in \N_0^m\setminus\{0\}$
	for $1\leq i\leq \ell$, 
	and the $1-\zeta_i\bfx^{\bfn_i}$'s
	are $\ell$ distinct irreducible
	polynomials.
		\item [(c)] The coefficients 
		$(a_{\bfn})_{\bfn\in\N_0^m}$
		belong to a finite set.	
	\end{itemize}
	\end{theorem}
	
	By specialization arguments, we have the
	following extension of the theorem by Bell-Chen:
	\begin{corollary}\label{cor:B-C structure}
	Let $K$, $m$, and $f$ be as in 
	Theorem~\ref{thm:new B-C}. 
	If the coefficients of $f$ belong to a finite
	set then parts (a) and (b) of 
	Theorem~\ref{thm:new main 1} hold.
	\end{corollary}
	
%	When $K=\Qbar$, Theorem~\ref{thm:referee main 1} is more
%	general than Theorem~\ref{thm:B-C} in the
%	sense that it allows power series 
%	whose coefficients belong to an infinite set.
%	In fact, we will prove an effective version of
%	the rationality of $f$
%	(see Theorem~\ref{thm:main 1 effective}). 
%	This effective version
%	together with a standard specialization argument
%	also imply Theorem~\ref{thm:B-C} (including
%	a description of the denominator as in
%	Theorem~\ref{thm:referee main 1}) for an arbitrary
%	$K$ of characteristic $0$. 

	Note that the condition \eqref{eq:thm 1}
	excludes rational 
	functions
	such as 
	$\displaystyle\frac{1}{1-2x}=\sum_{n\geq 0}2^nx^n$.  In 
	fact, the coefficients
	of a rational function have the form 
	$P_1(n)\alpha_1^n+\ldots+P_k(n)\alpha_k^n$
	and the logarithmic height is comparable to $n$ (unless all the $\alpha_i$'s are
	root of unity). Our next result proves that if a power series is D-finite and 
	its coefficients ``look like'' the coefficients of a rational
	function then the series is indeed rational. In fact,
	we will consider the above form
	$P_1(n)\alpha_1^n+\ldots+P_k(n)\alpha_k^n$ in which
	the polynomials
	$P_i$ can vary according to $n$ as long as 
	their degrees are
	bounded and their coefficients
	belong to a fixed number field and 
	have small heights compared to $\log(n)$:
	
	\begin{theorem}\label{thm:new main 2}
	Let $d\in \N_0$, $k\in\N$, and 
	$\alpha_1,\ldots,\alpha_k\in \Qbar^*$. 
	Let $K$ be a number field. For $n\geq 0$, let
	$a_n$ be of the form:
	$$a_n=(c_{n,1,0}+c_{n,1,1}n+\ldots+c_{n,1,d}n^d)\alpha_1^n+\ldots+(c_{n,k,0}+c_{n,k,1}n+\ldots+c_{n,k,d}n^d)\alpha_k^n$$
	such that $c_{n,i,j}\in K$ for $1\leq i\leq k$
	and $0\leq j\leq d$, 
	and
	$\displaystyle\lim_{n\to\infty}\frac{\displaystyle\max_{i,j}h(c_{n,i,j})}{\log n}=0$.
	If $f(x)=\displaystyle\sum_{n\geq 0}a_nx^n$ is D-finite then $f$ is rational. 
	\end{theorem}
	
	Roughly speaking, Theorem~\ref{thm:new main 1} 
	treats D-finite
	power series in which the heights of the coefficients grow
	very slowly while Theorem~\ref{thm:new main 2} considers
	those where the coefficients are similar to those of a typical rational functions 
	(and hence $h(a_n)$ is approximately linear in $n$). 
	We now consider D-finite series in which $h(a_n)$ can be large.
	The typical example is the exponential function
	$\exp(x)=\displaystyle\sum_{n\geq 0}a_nx^n$
	with $h(a_n)=\log(n!)\sim n\log(n)$. Our next result shows that 
	the heights of the coefficients of a univariate D-finite 
	power series cannot go beyond the function
	$n\log(n)$: 
	\begin{theorem}\label{thm:new main 3}
	Let $f(x)=\displaystyle\sum_{n\geq 0} a_nx^n\in\Qbar[x]$
	be D-finite. 
	For each $n\geq 0$, we consider the affine point
	$h(a_0,\ldots,a_n)$ and its Weil height. 
	We have:
	\begin{equation}	
	\limsup_{n\to\infty}\frac{h(a_n)}{n\log n}\leq \limsup_{n\to\infty}\frac{h(a_0,\ldots,a_n)}{n\log n}<\infty.
	\end{equation}
	\end{theorem}

%	In particular, the sequence
%	$(h(a_n))_{n\geq 0}$ cannot be 
%	``comparable'' to the sequence
%	$(1+\epsilon)^n$ for some $\epsilon>0$. 
%	We are grateful to the referee for suggesting that
%	a better upper bound than $(1+\epsilon)^n$
%	is possible. In fact, such an improvement of 
%	Theorem~\ref{thm:main 3 new} as well as 
%	certain improvements of
%	the lower bound $\log n$
%	in 
%	Theorem~\ref{thm:referee main 1} at least for univariate
%	D-finite series are the subject of an ongoing joint work
%	with Jason Bell and Umberto Zannier.

	The organization of this paper is
	as follows. In the next section, 
	we give a definition of the Weil height $h$ and various
	results needed for the proofs of the above theorems. Then 
	we  
	prove 
	Theorem~\ref{thm:new main 1} and 
	present specialization arguments for 
	Corollary~\ref{cor:B-C structure}. After that, we prove 
	Theorem~\ref{thm:new main 2} and 
	Theorem~\ref{thm:new main 3}. 
	
	\textbf{Acknowledgements.}
	%We thank the anonymous referee for helpful comments. 
	The first-named author is partially supported by an NSERC Discovery Grant. 
	The second-named author is partially supported by a start-up grant at the University of Calgary and an NSERC Discovery Grant.

	\section{Height}\label{sec:height}
	A large part of this section is taken from \cite{KMN} 
	which, in turn, follows 
	from earlier
	work of Evertse \cite{Eve84_OS} and
	Corvaja-Zannier \cite{CZ02_SN,CZ04_OT}. 	
  	Let $M_{\Q}=M_{\Q}^{\infty}\cup M_{\Q}^0$ where
  	$M_{\Q}^0$ is the set of $p$-adic valuations
  	and $M_{\Q}^{\infty}$
  	is the singleton consisting of the usual archimedean 
  	valuation. More generally, for every number field $K$,
  	write $M_K=M_K^\infty\cup M_K^0$ where
  	$M_K^\infty$ is the set of archimedean places and 
  	$M_K^0$ is the set of finite places. For every
  	$w\in M_K$, let $K_w$ denote the completion of 
  	$K$ with respect to $w$ and denote
  	$d(w)=[K_w:\Q_v]$ where $v$ is the restriction of
  	$w$ to $\Q$. Following \cite[Chapter~1]{BG06_HI}, 
  	for every $w\in M_K$ restricting to $v$ on $\Q$,
  	we normalize $\vert \cdot\vert_w$ as follows:
  	$$\vert x\vert_w = \vert \Norm_{K_w/\Q_v}(x) \vert_v^{1/[K:\Q]}.$$
  	Let $m\in\N$, for every vector $\bfu=(u_0,\ldots,u_m)\in K^{m+1}\setminus\{\mathbf 0\}$ and $w\in M_K$,
  	let $\vert \bfu\vert_w:=\displaystyle\max_{0\leq i\leq m} \vert u_i\vert_w$.
  	For $P \in \bP^m(\Qbar)$, let $K$ be a number
  	field such that $P$ has a representative 
  	$\bfu\in K^{m+1}\setminus\{\mathbf 0\}$
  	and define:
  	$$H(P)=\prod_{w\in M_K} \vert \bfu\vert_w.$$
  	Define $h(P)=\log (H(P))$. For $\alpha\in \Qbar$, 
  	write $H(\alpha)=H([\alpha:1])$
  	and $h(\alpha)=\log(H(\alpha))$. The following properties of
  	the height function are well-known \cite[Proposition~1.2]{Zan18_BO}:
  	\begin{proposition}\label{prop:height properties}
  	\begin{itemize}
  	\item [(a)] For every $a\in \Qbar^*$ and $m\in\Z$, $h(a^m)=\vert m\vert h(a)$.
  	\item [(b)] For every $r\in \N$ and $a_1,\ldots,a_r\in\Qbar$, $h(a_1+\ldots+a_r)\leq h(a_1)+\ldots+h(a_r)+\log r$.
  	\item [(c)] For every $a,b\in\Qbar$, $h(ab)\leq h(a)+h(b)$. Hence if $b\neq 0$ then $h(ab)\geq h(a)-h(b)$.
  	\item [(d)] Let $P(t)=a_dt^d+\ldots+a_1t+a_0\in\Qbar[t]$. There
  	exist constants $C_0(d)$ and $C_1(d)$ depending only on $d$ such that
  	if $a_d\neq 0$ then
  	$$\vert h(P(\alpha))-dh(\alpha)\vert \leq C_1(d)\max_{0\leq i\leq d}h(a_i)+C_0(d)$$
  	for every $\alpha\in\Qbar$.
  	\end{itemize}
  	\end{proposition}
 	\begin{proof}
 	Parts (a), (b), and (c) are in any standard introduction to
 	Weil heights such as \cite[Part~B]{HS00_DG} or \cite[Chapters~1--2]{BG06_HI}. For part (d), see \cite[Proposition~6]{HS11_AQ} and
 	\cite[Remark~B.2.7]{HS00_DG}.
 	\end{proof}

  	Now we present an important application  of the
  	Subspace Theorem taken from \cite[Section~2]{KMN}. The Subspace
  	Theorem is one of the milestones of diophantine geometry in
  	the last 50 years. The first version was obtained by 
  	Schmidt \cite{Sch70_SA} and further versions were obtained
  	by Schlickewei and Evertse \cite{Sch90_TQ,Eve96_AI,ES02_AQ}.
 	In the following application, a sublinear function means
 	a function $F:\ \N\rightarrow (0,\infty)$
 	such that 
 	$\displaystyle\lim_{n\to\infty}\frac{F(n)}{n}=0$. 
 	Let $k\in\N$, a tuple of non-zero algebraic 
 	numbers $(\alpha_1,\ldots,\alpha_k)$ is said to
 	be non-degenerate
 	if $\alpha_i/\alpha_j$ is not a root of unity
 	for $i\neq j$. We have:
 	
 	\begin{proposition}\label{prop:general SML}
  	Let $k\in\N$, let $(\alpha_1,\ldots,\alpha_k)$ be a non-degenerate 
  	tuple of non-zero algebraic numbers, let $F$ be a sublinear function, and let $K$ be a number field. Then there are 
  	only finitely many tuples $(n,b_1,\ldots,b_k)\in \N\times (K^*)^k$ satisfying:
  	$$b_1\alpha_1^n+\ldots+b_k\alpha_k^n=0\ \text{and}\ \max_{1\leq i\leq k}h(b_i)<f(n).$$
  	\end{proposition}
  	\begin{proof}
  	 This follows from a result of Evertse 
  	 \cite[Theorem~1]{Eve84_OS}.
  	 For more details, see \cite[Section~2]{KMN}.
  	\end{proof}
 	
% 	The next application is a slight modification of
%	\cite[Lemma~1]{CZ04_OT} which, in turn, is a 
%	consequence of \cite[Theorem~2]{Eve84_OS}.
%  	\begin{proposition}\label{prop:CZLemma1}
%  	Let $k\in\N$, let $K$ be
%  	a number field, let $S$ be a finite subset of $M_K$ containing $M_K^{\infty}$, and let $\lambda_1,\ldots,\lambda_k$ be non-zero elements of $K$. Fix $w\in S$
%  	and $\epsilon>0$. Let $\scrS$ be an infinite set of
%  	solutions $(u_1,\ldots,u_k,b_1,\ldots,b_k)$
%  	of the inequality:
%  	$$\left\vert \displaystyle\sum_{j=1}^k\lambda_jb_ju_j\right\vert_w\leq \max\{\vert b_1u_1\vert_w,\ldots,\vert b_ku_k\vert_w\}(\displaystyle\prod_{j=1}^k H(b_j))^{-k-1-\epsilon}H([u_1:\ldots:u_k:1])^{-\epsilon}$$
%  	where $u_j$ is an $S$-unit and $b_j\in K^*$ for $1\leq j\leq k$. Then there exists a non-trivial
%  	linear relation of the form $c_1b_1u_1+\ldots+c_kb_ku_k=0$
%  	where $c_i\in K$ for $1\leq i\leq k$ satisfied
%  	by infinitely many elements of $\scrS$.
%  	\end{proposition} 
%  	\begin{proof}
%  	See \cite[Section~2]{KMN}.
%  	\end{proof}

	\section{Proofs of Theorem~\ref{thm:new main 1} and Corollary~\ref{cor:B-C structure}}\label{sec:proofs 1}
	
	We will refer to the following property of power series with
	algebraic coefficients throughout the paper:
	\begin{definition}
		Let $m\in\N$, $\bfx=(x_1,\ldots,x_m)$, and
		$f(\bfx)=\sum_{\bfn\in\N_0^m}a_{\bfn}\bfx^\bfn$.
		We say that $f$ satisfies property $\cP$ if:
		$$\lim_{\Vert\bfn\Vert\to\infty}\frac{h(a_{\bfn})}{\log\Vert\bfn\Vert}=0.$$	
	\end{definition}
	
	For the rest of this section, let $m\in\N$ and $\bfx=(x_1,\ldots,x_m)$.	
	The proof of Theorem~\ref{thm:new main 1} consists
	of three parts. The first part is to use properties of
	the Weil height to 
	establish rationality of $f$. The key idea
	is that the coefficients of $f$ satisfy certain
	linear recurrence relations with polynomial coefficients
	and the property $\cP$
	allows the polynomial coefficients
	to be the dominant terms in such relations. 
	In fact we will prove an effective version of part (a) 
	Theorem~\ref{thm:new main 1} which will be used in 
	the specialization arguments
	for the proof of Corollary~\ref{cor:B-C structure}.
	The second part of the proof is to prove part (b)
	by using the substitution 
	$(x_1,\ldots,x_m)=(t^{u_1},\ldots,t^{u_m})$
	for $u_1,\ldots,u_m\in\N$ in order to apply known
	results about univariate rational functions; 
	it turns out that this part has a surprising
	connection to the beautiful Manin-Mumford conjecture
	for tori in diophantine geometry.
	Finally, once we know that $f$ is a rational function
	whose denominator has the special form
	given in part (b), we can use induction and certain
	combinatorial arguments to finish the proof. We start with the
	following simple lemma:
	\begin{lemma}\label{lem:Lipshitz}
	Let $K$ be a field of characteristic $0$ and let
	$f(\bfx)=\sum_{\bfn\in\N_0^m}a_{\bfn}\bfx^{\bfn}[[\bfx]]$. 
	We have:
	\begin{itemize}
	\item [(a)] $f$ is D-finite over $K(\bfx)$ if and only if 
	$f$ satisfies a system of linear partial differential equations,
	one for each $i=1,\ldots,m$, of the form:
	$$\left(P_{i,d_i}(\bfx)\left(\frac{\partial}{\partial x_i}\right)^{d_i}+\ldots+P_{i,1}(\bfx)\frac{\partial}{\partial x_i}+P_{i,0}(\bfx)\right)f(\bfx)=0$$
	where $P_{i,j}(\bfx)\in K[\bfx]$ for every
	$0\leq j\leq d_i$ and $P_{i,d_i}(\bfx)\neq 0$.
	\item [(b)] Let $F$ be a field containing $K$. Then $f$ is
	D-finite over $F(\bfx)$ if and only if it is D-finite
	over $K(\bfx)$. 
	\end{itemize}
	\end{lemma}
	\begin{proof}
	Part (a) is \cite[Proposition~2.2]{Lip89_DF}; although the author
	stated it for $\C$, the proof works verbatim for
	an arbitrary field $K$ of characteristic $0$. For part (b),
	if $f$ is D-finite over $F(\bfx)$ then the coefficients
	of the $P_{i,j}$'s give a non-trivial solution
	over $F$ of a homogeneous system of (infinitely many)
	linear equations with coefficients in $K$. Hence this system
	must have a non-trivial solution over $K$ and this proves
	D-finiteness over $K(\bfx)$. 
	\end{proof}

	\subsection{Proof of part (a) of Theorem~\ref{thm:new main 1}}\label{subsec:Proof a}
	We now prove the following effective version
	of part (a) of Theorem~\ref{thm:new main 1}: 
	\begin{theorem}\label{thm:main 1 effective}
	Let $f(\bfx)=\displaystyle\sum_{\bfn\in\N_0^m}a_{\bfn}\bfx^{\bfn}\in \Qbar[[\bfx]]$ be D-finite. Assume that $f$ satisfies a system of linear partial differential equations,
	one for each $i=1,\ldots,m$, of the form:
	$$\left(P_{i,d_i}(\bfx)\left(\frac{\partial}{\partial x_i}\right)^{d_i}+\ldots+P_{i,1}(\bfx)\frac{\partial}{\partial x_i}+P_{i,0}(\bfx)\right)f(\bfx)=0$$
	where $P_{i,j}(\bfx)\in \Qbar[\bfx]$ for every
	$0\leq j\leq d_i$ and $P_{i,d_i}(\bfx)\neq 0$.
	Let $M$ be an upper bound on the heights of the coefficients
	and let $D$ be an upper bound on the total 
	degrees of all the $P_{i,j}$. Then there exist  
	effectively computable
	positive 
	constant $\delta$
	and $\eta$
	 depending only on 
	$m$, $M$, $D$, and $\displaystyle\max_{1\leq i\leq m}d_i$
	such that the following holds.
	If $N$ satisfies
	$\displaystyle\frac{h(a_{\bfn})}{\log \Vert \bfn\Vert}<\delta$
	for every $\bfn\in\N_0^m$ with 
	$\Vert\bfn\Vert\geq N$
	then $P_{1,n_1}\ldots P_{m,n_m}f$
	is a polynomial of total degree at most
	$N+\eta$.
	\end{theorem}

	For $1\leq i\leq m$, let $\bfe_{i}=(0,\ldots,0,1,0,\ldots,0)$ be the 
	$i$-th elementary basis vector in $\N_0^m$. For every
	$j\in \N$, let $B_j(x)=x(x-1)\ldots (x-j+1)\in \Z[x]$
	and let $B_0(x)=1$. So we have:
	$$\left(\frac{\partial}{\partial x_i}\right)^{j}f(\bfx)=
	\sum_{\bfn=(n_1,\ldots,n_m)\in\N_0^m}B_j(n_i)a_{\bfn}\bfx^{\bfn-j\bfe_i}.$$
	To prove Theorem~\ref{thm:main 1 effective}, we prove
	the following result that handles one linear partial
	differential equation at a time:
	\begin{proposition}\label{prop:1 differential eq}
	Let $f(\bfx)=\displaystyle\sum_{\bfn\in\N_0^m}a_{\bfn}\bfx^{\bfn}\in \Qbar[[\bfx]]$ and fix $i\in\{1,\ldots,m\}$. Assume that $f$ satisfies the linear partial differential equation:
	\begin{equation}\label{eq:1 differential eq}
	\left(P_{i,d_i}(\bfx)\left(\frac{\partial}{\partial x_i}\right)^{d_i}+\ldots+P_{i,1}(\bfx)\frac{\partial}{\partial x_i}+P_{i,0}(\bfx)\right)f(\bfx)=0
	\end{equation}
	where $P_{i,j}(\bfx)\in \Qbar[\bfx]$ for every
	$0\leq j\leq d_i$ and $P_{i,d_i}(\bfx)\neq 0$.
	Let $M_i$ be an upper bound on the height of the coefficients
	and let
	$D_i$ be an upper bound on the total degrees of
	the $P_{i,j}$'s for $0\leq j\leq d_i$.
	Let $\epsilon_i>0$, then there exist effectively computable
   positive constants $\delta_i$
   and $\eta_i$ depending only on $m$, $M_i$, $D_i$, $d_i$, and
   $\epsilon_i$  
    such that the following holds.
   If $N$ satisfies 
   $\displaystyle\frac{h(a_{\bfn})}{\log\Vert\bfn\Vert}<\delta_i$
   for every $\bfn\in\N_0^m$ with $\Vert \bfn\Vert\geq N$
   then for every $\bfr=(r_1,\ldots,r_m)\in\N_0$,
   if $\Vert \bfr\Vert \geq N+\eta_i$ and 
   $r_i\geq \epsilon_i\Vert \bfr\Vert$ 
   then the coefficient of $\bfx^{\bfr+d_i\bfe_i}$
   in $P_{i,d_i}f$ is $0$.  
	\end{proposition}
	\begin{proof}
	If $d_i=0$ then $P_{i,0}f=0$ and there is nothing to prove,
	so we may assume $d_i>0$. 
	For $0\leq j\leq d_i$, let $S_{i,j}\subset\N_0^m$ be the ``support'' of
	$P_{i,j}$; this means the finite set of the 
	multi-degrees of
	monomials having non-constant coefficients in $P_{i,j}$. 
	For $0\leq j\leq d_i$, write:
	$$P_{i,j}(\bfx)=\sum_{\bfn\in S_{i,j}}p_{i,j,\bfn}\bfx^{\bfn}.$$
	
	Let $\bfr=(r_1,\ldots,r_m)\in\N_0^m$, the coefficient of
	$\bfx^\bfr$ in the left-hand side of 
	\eqref{eq:1 differential eq} is: 
	\begin{equation}\label{eq:1 differential eq lhs}
	\sum_{j=0}^{d_i}\sum_{\bfn=(n_1,\ldots,n_m)\in S_{i,j}}p_{i,j,\bfn}B_j(r_i+j-n_i)a_{\bfr+j\bfe_i-\bfn}=0;
	\end{equation}  
	note that our convention here is to put
	$a_{\bfu}=0$ if $\bfu\in \Z^{m}\setminus\N_0^m$.  
	Since $\Vert\bfn\Vert\leq D_i$ for every $\bfn\in S_{i,j}$,
	there exists a constant $C_2$  depending only
	on $d_i$ and $D_i$ such that 
	for every $0\leq j\leq d_i$ and every
	$\bfn=(n_1,\ldots,n_m)\in S_{i,j}$, 
	$B_j(r_i+j-n_i)$ is a polynomial of 
	degree $j$ in $r_i$ and the heights of its coefficients  
	are bounded above by $C_2$.  
	
	Now assume that $\Vert \bfr\Vert\geq\max\{N+d_i+D_i,2\}$ so that 
	for every
	$0\leq j\leq d_i$ and every $\bfn\in S_{i,j}$, the vector $\bfr+j\bfe_i-\bfn$
	is either in $\Z^m\setminus\N_0^m$ or the sum of its coordinates
	is at least $N+d_i$ and we have:
	$$h(a_{\bfr+j\bfe_i-\bfn})\leq \delta_i\log(\Vert \bfr\Vert+j-\Vert\bfn\Vert)\leq \delta_i\log(2\Vert \bfr\Vert)\leq 2\delta_i\log\Vert\bfr\Vert.$$ 
	Observe that the cardinality of each $S_{i,j}$
	is at most $(D_i+1)^m$. By gathering the coefficients of
	common powers of $r_i$ and using
	Proposition~\ref{prop:height properties}, 
	we can write the left-hand side
	of \eqref{eq:1 differential eq lhs} as:
	\begin{equation}\label{eq:1 differential eq alpha}
	\alpha_{d_i}r_i^{d_i}+\alpha_{d_i-1}r_i^{d_i-1}+\ldots+\alpha_0=0
	\end{equation}
	where
	$\alpha_{d_i}=\displaystyle\sum_{\bfn\in S_{i,d_i}}p_{i,d_i,\bfn}a_{\bfr+d_i\bfe_i-\bfn}$
	and the following holds. There exist constants $C_3$ and 
	$C_4$ depending only on $m$, $M_i$, $D_i$, and $d_i$
	such that 
	$h(\alpha_j)\leq C_3\delta_i\log\Vert \bfr\Vert+C_4$
	for $j=0,\ldots,d_i$. By Proposition~\ref{prop:height properties}(d), we have $C_5$ and $C_6$ such that
	if $\alpha_{d_i}\neq 0$ then: 
	\begin{align}
	\begin{split}\label{eq:C5 C6}
	C_5\delta_i\log\Vert\bfr\Vert+C_6&\geq \vert h(\alpha_{d_i}r_i^{d_i}+\alpha_{d_i-1}r_i^{d_i-1}+
	\ldots+\alpha_0)-d_ih(r_i)\vert\\
	&=d_i\log(r_i)\\
	&\geq d_i\log(\epsilon_i \Vert\bfr\Vert)
	\end{split}
	\end{align}
	where the last inequality is under the further assumption
	that $r_i\geq \epsilon_i\Vert\bfr\Vert$. However 
	\eqref{eq:C5 C6} cannot hold when $\delta_i$ 
	is sufficiently
	small and $\Vert \bfr\Vert$ is sufficiently large,
	for instance when $C_5\delta_i\leq d_i/2$
	and $\Vert r\Vert >e^{2C_6}/\epsilon_i^2$. Hence under this
	further assumption, we must have
	$\alpha_{d_i}=0$. Notice
	that $\alpha_{d_i}=\displaystyle\sum_{\bfn\in S_{i,d_i}}p_{i,d_i,\bfn}a_{\bfr+d_i\bfe_i-\bfn}$
	is exactly the coefficient of $\bfx^{\bfr+d_i\bfe_i}$
	in $P_{i,d_i}f$
	and we finish the proof. 
	\end{proof}
	
	Proposition~\ref{prop:1 differential eq} is the
	key step in our proof of Thereom~\ref{thm:main 1 effective}:
	\begin{proof}[Proof of Theorem~\ref{thm:main 1 effective}]
	We apply Proposition~\ref{prop:1 differential eq} for
	each $i=1,\ldots,m$ with
	$\epsilon_i=1/2m$, 
	let $\delta$ be the minimum of the resulting
	$\delta_i$'s, and let $\eta'$ be the maximum
	of the resulting $\eta_i$'s. We now take:
	$$\eta'':=\tilde{\eta}+(2m-1)(d_1+\ldots+d_m).$$
	
	Let $\bfr\in\N_0^m$ with $\Vert\bfr\Vert>N+\eta''$.
	There exists $i\in\{1,\ldots,m\}$ such that
	$r_i\geq \Vert \bfr\Vert/m$. Hence the vector
	$\bfr':=\bfr-d_i\bfe_i$ satisfies
	$r'_i\geq \Vert \bfr'\Vert/2m$ and
	$\Vert\bfr'\Vert> N+\tilde{\eta}\geq N+\eta_i$. By
	Proposition~\ref{prop:1 differential eq}, 
	the coefficient of $\bfx^{\bfr}$ in
	$P_{i,d_i}f$ is zero. Therefore, if we choose
	$\eta:=\eta''+(m-1)D$
	then $P_{1,d_1}\ldots P_{m,d_m}f$
	is a polynomial of total degree at most
	$N+\eta$. 
	\end{proof}
	
	\subsection{Proof of part (b) of Theorem~\ref{thm:main 1 effective}}
	We will use the following simple result for univariate 
	rational functions:
	\begin{proposition}\label{prop:denominator 1 variable}
	Let $G(t)=\sum_{n\geq 0}g_nx^n\in \Qbar[[t]]$ 
	be a rational function that is not a polynomial. Assume
	$h(g_n)=o(n)$
	then every root of the denominator of $G$ is
	a root of unity. Moreover if 
	$$\limsup_{n\to\infty}\frac{h(g_n)}{\log n}\leq L<\infty$$
	then every root of the denominator of 
	$G$ has multiplicity at most $L+1$. 
	\end{proposition}
	\begin{proof}
	Let $\alpha_1,\ldots,\alpha_{\ell}$ be all the (distinct)
	 roots 
	of the 
	denominator of $G$; we have 
	$\alpha_i\in \Qbar^*$ for every $i$. 
	Then there exist
	$P_1(X),\ldots,P_{\ell}(X)\in\Qbar[X]\setminus\{0\}$
	such that for all sufficiently large $n$, we have:
	$$g_n=P_1(n)\alpha_1^n+\ldots+P_{\ell}(n)\alpha_{\ell}^n.$$
	
	For the first assertion, assume $h(g_n)=o(n)$ and we prove that all the $\alpha_i$'s are roots of unity. This can be done easily using induction on $r:=\ell+\sum_{i=1}^{\ell}\deg(P_i)$ and working with the sequence $g_{n+1}-\alpha_{\ell}g_n$ which lowers the value of $r$.
	
	For the second assertion, let $D$ denote the maximum of the degrees of the $P_i$'s.  Then for $n$ belonging to 
	an appropriate arithmetic progression, $g_n=\sum_{i=1}^{\ell}P_i(n)\alpha_i^n$ is a polynomial in $n$ with degree 
	$D$. Hence $D\leq L$ and this finishes the proof.
	\end{proof}
	
	We will use the following version of the Manin-Mumford
	conjecture for tori. For $n\geq 1$, by a torsion coset
	of $\Gmn$, we mean a torsion translate of an algebraic
	subgroup. For a closed subvariety $V$ of $\Gmn$,
	a torsion coset in $V$ means a torsion coset of
	$\Gmn$ that is contained in $V$. 
	\begin{theorem}\label{thm:Manin-Mumford}
	Let $n\geq 1$ and let $V$ be a closed subvariety of $\Gmn$
	defined over $\C$. Then the following hold:
	\begin{itemize}
		\item [(a)]  Every torsion coset in $V$ is contained
		in a maximal torsion coset in $V$.
		\item [(b)]  There are only finitely many maximal torsion cosets in $V$ and their union is the Zariski closure
		of torsion points in $V$.
	\end{itemize}
	\end{theorem}
	\begin{proof}
	This is given in \cite[Chapter 3]{BG06_HI}
	following earlier work of Laurent \cite{Lau84_ED}, 
	Bombieri-Zannier \cite{BZ95_AP}, and Schmidt \cite{Sch96_HO}. In fact, the number of maximal torsion
	cosets can be bounded by an explicit expression involving
	only $n$ and the maximum of the degrees of polynomials
	defining $V$. 
	\end{proof}

	\begin{lemma}\label{lem:nonzero substitution}
	Let $P(x_1,\ldots,x_n)\in\C[x_1,\ldots,x_n]\setminus\{0\}$,
	then there exist finitely many proper vector subspaces
	$W_1,\ldots,W_k\subsetneq \Q^n$ such that for
	every $(u_1,\ldots,u_n)\in \N^n$ outside
	$\bigcup_{j=1}^k W_j$ we have
	$P(t^{u_1},\ldots,t^{u_n})$ is a non-zero polynomial in
	$t$.
	\end{lemma}
	\begin{proof}
		If we do the substitution $x_i=t^{u_i}$ and 
	get $P(t^{u_1},\ldots,t^{u_n})=0$, 
	then two distinct monomials in $P(x_1,\ldots,x_n)$ 
	yield the same $t^k$ and this gives rise to 
	a non-trivial linear relation among the $u_i$'s.
	\end{proof}
	
	\begin{proof}[Proof of part (b) of Theorem~\ref{thm:new main 1}]
	We have proved that $f$ is a rational function. Suppose 
	that $f$ is not a polynomial and write
	$\displaystyle f=\frac{F}{G}$
	where $F$ and $G$ are coprime polynomials in $\Qbar[x_1,\ldots,x_m]$ and $G$ is non-constant. 
	We first prove that every irreducible factor of $G$ has
	the form $1-\zeta\bfx^{\bfn}$ where
	$\zeta$ is a root of unity
	and $\bfn\in\N_0^m$.
	
	Since 
    the property
	$\cP$ still holds	
	after replacing $f$ by its product with a polynomial, 
	we may assume that $G$ is irreducible. 
	Fix an embedding of $\Qbar$ into $\C$, the 
	condition $h(a_{\bfn})=o(\log\Vert\bfn\Vert)$
	implies that $f$ is convergent in the polydisc
	$\cD$ given by $\vert x_i\vert<1$ for $1\leq i\leq m$. 
	For a polynomial $P\in\C[x_1,\ldots,x_m]$, let $\cZ(P)$
	denote the zero set of $P$.
	If 
	$G(0,\ldots,0)=0$ then $\cZ(G)\cap \cD$
	is contained in $\cZ(F)\cap \cD$ since
	$F=fG$ as analytic functions on $\cD$. 
	But this is impossible since $\cZ(F)\cap \cZ(G)$
	has strictly smaller dimension than $\cZ(F)$.
	Hence $G(0,\ldots,0)\neq 0$.
	
	Since $G$ is not one of the coordinate functions $x_i$'s,
	the closed subvariety $V$ of $(\C^*)^m$ defined by
	$G=0$ has dimension $m-1$ and our goal is to prove that
	$V$ is a torsion coset. Assume that each of the finitely 
	many maximal torsion coset in $V$ has codimension 
	at least $2$
	and we will arrive at a contradiction. Each such maximal 
	torsion coset satisfies at least $2$ independent 
	equations of
	the form:
	$$x_1^{\gamma_1}\ldots x_m^{\gamma_m}=1,$$
	$$x_1^{\delta_1}\ldots x_m^{\delta_m}=1.$$
	Therefore we can eliminate $x_m$ if necessary 
	to conclude that
	the maximal torsion coset is contained in
	the subgroup defined by an equation of the form:
	$$x_1^{\kappa_1}\ldots x_{m-1}^{\kappa_{m-1}}=1.$$
	Since there are only finitely many maximal torsion cosets,
	we obtain a finite set $\scrK$ of non-zero vectors
	in $\Z^{m-1}$ such that for every torsion
	point $(\xi_1,\ldots,\xi_m)\in V$,
	there is a vector $(\kappa_1,\ldots,\kappa_{m-1})\in \scrK$
	satisfying:
	\begin{equation}\label{eq:xi}
	\xi_1^{\kappa_1}\ldots\xi_{m-1}^{\kappa_{m-1}}=1.
	\end{equation}
	
	Let $K$ be a number field containing the coefficients
	of $F$ and $G$. By relabelling the $x_i$'s when necessary,
	we may assume that $G$ has the form:
	$$G=G_0+G_1x_m+\ldots+G_dx_m^d$$
	where $d\geq 1$, each $G_i$ is in
	$K[x_1,\ldots,x_{m-1}]$, and 
	$G_d\neq 0$. Since $F$ and $G$ are coprime, there exist  
	polynomials
	$\tilde{F},\tilde{G},\tilde{H}$ in
	$K[x_1,\ldots,x_{m-1}]$ with
	$\tilde{H}\neq 0$ such that
	$$\tilde{F}F+\tilde{G}G=\tilde{H}.$$
	By Lemma~\ref{lem:nonzero substitution}, 
	there is a union $W$ of finitely many proper subspaces
	of $\Q^{m-1}$ such that for 
	every $(u_1,\ldots,u_{m-1})\in \N^{m-1}\setminus W$,
	we have: 
	\begin{equation}\label{eq:tildeF tildeG tildeH}	
	G_d(t^{u_1},\ldots,t^{u_{m-1}})\tilde{H}(t^{u_1},\ldots,t^{u_{m-1}})\neq 0.
	\end{equation}
	
	By adding to $W$ the subspaces of $\Q^{m-1}$
	each of which is the orthogonal complement to
	some $(\kappa_1,\ldots,\kappa_{m-1})\in \scrK$, we may
	assume the additional property that
	$$u_1\kappa_1+\ldots+u_{m-1}\kappa_{m-1}\neq 0$$
	for every $(\kappa_1,\ldots,\kappa_{m-1})\in \scrK$.   
	
	Fix one such $(u_1,\ldots,u_{m-1})$ and let 
	$$B=\max\{\vert u_1\kappa_1+\ldots+u_{m-1}\kappa_{m-1}\vert:\ (\kappa_1,\ldots,\kappa_{m-1})\in\scrK\}.$$
	Let $S$ be a finite subset of $M_K$ containing 
	$M_K^{\infty}$ such that the ring of $S$-integers
	$\cO_{K,S}$ is a UFD and the coefficients of 
	$F$ and $G$ are in $\cO_{K,S}$. From $F=fG$ and
	the fact that $\cO_{K,S}[x_1,\ldots,x_m]$ is a UFD, 
	we conclude
	that the coefficients of $f$ are in $\cO_{K,S}$ too.
	Let $u_m\in\N$ that will be chosen to be sufficiently
	large. 
	
	Consider the following rational
	function in $t$:
	$$\frac{F(t^{u_1},\ldots,t^{u_m})}{G(t^{u_1},\ldots,t^{u_m})}=f(t^{u_1},\ldots,t^{u_m})=:\sum_{n\geq 0}\tau_nt^n$$
	where $\tau_n=\sum_{\bfn}a_{\bfn}$
	in which $\bfn$ ranges over all 
	$\bfn=(n_1,\ldots,n_m)$ with
	$n_1u_1+\ldots+n_nu_m=n$; there are $O(n^{m-1})$ such
	$\bfn$'s. Equation \eqref{eq:tildeF tildeG tildeH}
	implies  
	\begin{equation}\label{eq:gcd mid tildeH}
	\gcd(F(t^{u_1},\ldots,t^{u_m}),G(t^{u_1},\ldots,t^{u_m})) 
	\mid \tilde{H}(t^{u_1},\ldots,t^{u_{m-1}}).
	\end{equation} 
	Hence when $u_m$ is sufficiently large
	so that 
	$$\deg(G(t^{u_1},\ldots,t^{u_m}))=du_m+\deg(\tilde{G}(t^{u_1},\ldots,t^{u_{m-1}}))>\deg(\tilde{H}(t^{u_1},\ldots,t^{u_{m-1}})),$$
	$f(t^{u_1},\ldots,t^{u_m})$ is not a polynomial
	and its denominator is:
	\begin{equation}\label{eq:denominator expression}
	\frac{G(t^{u_1},\ldots,t^{u_m})}{\gcd(F(t^{u_1},\ldots,t^{u_m}),G(t^{u_1},\ldots,t^{u_m}))}.
	\end{equation}
	 
	Since the $a_{\bfn}$'s are in $\cO_{K,S}$ and $h(a_{\bfn})=o(\log\Vert\bfn\Vert)$, we have:
	\begin{itemize}
		\item $\tau_n$ is in $\cO_{K,S}$ for every $n$.
		\item $\vert\tau_n\vert_v\leq \max\{\vert a_{\bfn}\vert_v:\ \Vert\bfn\Vert\leq n\}=n^{o(1)}$ for
		every $v\in S\cap M_K^0$.
		\item $\vert\tau_n\vert_v\leq n^{m-1+o(1)}$ for
		every $v\in S\cap M_K^{\infty}$.	
	\end{itemize}
	Therefore $h(\tau_n)\leq (m-1+o(1))\log n$. 
	Proposition~\ref{prop:denominator 1 variable}
	implies that the denominator of
	$f(t^{u_1},\ldots,t^{u_m})$ has the form
	\begin{equation}\label{eq:exponents e_i}
	(t-\zeta_1)^{e_1}\ldots (t-\zeta_{\ell})^{e_{\ell}}
	\end{equation}
	where $\ell\geq 1$, the
	$\zeta_i$'s are $\ell$ distinct roots of unity,
	and $1\leq e_i\leq m$.
	
	From the expressions \eqref{eq:denominator expression}
	and \eqref{eq:exponents e_i} for the denominator
	of $f(t^{u_1},\ldots,t^{u_m})$
	and \eqref{eq:gcd mid tildeH}, we have:
	$$G(t^{u_1},\ldots,t^{u_m})=\gcd(F(t^{u_1},\ldots,t^{u_m}),G(t^{u_1},\ldots,t^{u_m}))\prod_{j=1}^{\ell}(t-\zeta_j)^{e_j}$$
	and
	\begin{align*}
	\ell m\geq\sum_{j=1}^{\ell}e_j&=\deg(G(t^{u_1},\ldots,t^{u_m})-\deg(\gcd(F(t^{u_1},\ldots,t^{u_m}),G(t^{u_1},\ldots,t^{u_m})))\\
	&\geq du_m+\deg(\tilde{G}(t^{u_1},\ldots,t^{u_{m-1}}))-\deg(\tilde{H}(t^{u_1},\ldots,t^{u_{m-1}})).
	\end{align*}
	
	With a sufficiently large $u_m$, we have $\ell>B$.
	Choose a $\zeta_i$ with $1\leq i\leq \ell$ such that
	$\zeta_i$ has order at least
	$\ell$. Since
	$G(\zeta_i^{u_1},\ldots,\zeta_i^{u_m})=0$, the point
	$(\zeta_i^{u_1},\ldots,\zeta_i^{u_m})$ is a torsion
	point of $V$. But we have
	$$(\zeta_i^{u_1})^{\kappa_1}\ldots 
	(\zeta_i^{u_m})^{\kappa_m}=\zeta_i^{u_1\kappa_1+\ldots+u_m\kappa_{m-1}}\neq 1$$
	for every $(\kappa_1,\ldots,\kappa_{m-1})\in\scrK$
	since $0<\vert u_1\kappa_1+\ldots+u_{m-1}\kappa_{m-1}\vert\leq B$ which is less than the order of $\zeta_i$.
	This contradicts \eqref{eq:xi}. Therefore
	$V$ itself is a torsion coset. Since
	$G(0,\ldots,0)\neq 0$, we conclude
	that $G$ has the form
	$1-\zeta\bfx^{\bfn}$.

	Now we no longer assume that $G$ is irreducible.
	The above arguments prove that every irreducible factor
	of $G$ has the form $1-\zeta\bfx^{\bfn}$. 
	To finish the proof of part (b) of 
	Theorem~\ref{thm:new main 1}, it 
	remains to show that every irreducible factor
	of $G$ has multiplicity $1$. As before, by 
	considering the product of $f$ with a polynomial, 
	we may assume that $G=(1-\zeta\bfx^{\bfn})^r$ 
	in which $r\in\N$ and $1-\zeta\bfx^{\bfn}$ is irreducible. 
	Let $s$ denote the order of $\zeta$, then we can write:
	$$f=\frac{F}{(1-\zeta\bfx^{\bfn})^r}=\frac{P}{(1-\bfx^{\bfn'})^r}$$
	where $\bfn':=s\bfn$ and $P(\bfx)\in \Qbar[\bfx]$
	that is not divisible by $1-\bfx^{\bfn'}$. Assume that
	$r\geq 2$ and we will arrive at a contradiction.
	Write 
	$$P(\bfx)=\sum_{\bfk\in S(P)}p_{\bfk}x^{\bfk}$$
	where $S(P):=\{\bfk\in\N_0^m:\ p_{\bfk}\neq 0\}$
	is the support of $P$. On $\N_0^m$, define the
	equivalence relation: $\bfk\sim \bfk'$
	if and only if $\bfk-\bfk'\in \Z\bfn$.
	The equivalence class of
	$\bfk\in\N_0^m$ is denoted $\overline{\bfk}$.
	Fix $\bfk^*\in S(P)$ and let $\alpha\in \N$ be
	sufficiently large, by computing the Taylor
	series of 
	$\displaystyle \frac{P}{(1-\bfx^{\bfn'})^r}$
	directly, we have that the coefficient
	$a_{\bfk^*+\alpha\bfn}$ is a polynomial of
	degree $r-1$ in $\alpha$ whose leading coefficient
	has the form
	$$c\sum_{\bfk\in \overline{\bfk^*}\cap S(P)}p_{\bfk}$$
	where $c$ is a non-zero constant.
	By the assumption on the height of the coefficients
	of $f$, we must have:
	$$\sum_{\bfk\in \overline{\bfk^*}\cap S(P)}p_{\bfk}=0.$$
	Since this is true for every $\bfk^*\in S(P)$, we have
	that $P$ is divisible by $1-\bfx^{\bfn'}$, contradiction.
	Hence $r=1$ and we finish the proof.
	\end{proof}

	\subsection{Proof of part (c) of Theorem~\ref{thm:new main 1}}
	We use induction on the number of variables $m$. The case
	$m=1$ follows from 
	Proposition~\ref{prop:denominator 1 variable}.
	Now consider $m\geq 2$ and assume that the conclusion
	holds for all power series with less than $m$ variables.
	
	If $x_m$ does not appear in the denominator of 
	$f$ then we can write $f$
	as a \emph{finite} sum:
	$$\sum_{n\geq 0}x_m^n f_n(x_1,\ldots,x_{m-1})$$
	in which each $f_n$ is D-finite and 
	satisfies property $\cP$ for power series 
	in $m-1$ variables. Then we are done by the induction
	hypothesis. So we may assume that $x_m$ appears in
	the denominator of $f$. By part (b), we can write:
	$$f=\frac{P(x_1,\ldots,x_m)}{Q(x_1,\ldots,x_{m-1})\prod_{i=1}^{\ell}(1-\zeta_i\bfx^{\bfn_i})}$$
	where the $1-\zeta_i\bfx^{\bfn_i}$'s are all
	the irreducible factors of the denominator of $f$
	in which $x_m$ appears and $Q(x_1,\ldots,x_{m-1})\in\Qbar[x_1,\ldots,x_{m-1}]$
	is the product of the remaining irreducible factors.
	
	Write $\bfn_i=(n_{i,1},\ldots,n_{i,m})$ for $1\leq i\leq \ell$, hence $n_{i,m}>0$ for every $i$. 
	Denote $r_{i,j}=n_{i,j}/n_{i,m}$ for
	$1\leq i\leq \ell$ and $1\leq j\leq m-1$. 
	Consider the change of variables: $x_m=y_m$, $x_{m-1}=y_{m-1}$, and 
	$$x_j=y_j(y_{j+1}\ldots y_{m-1})^{u_j}$$
	for $1\leq j\leq m-2$ where the $u_j$'s will be chosen as follows. We start with a sufficiently
	large $u_{m-2}\in\N$, then $u_{m-3}\in\N$ with
	a sufficiently large $u_{m-3}/u_{m-2}$,
	and so on until $u_1\in\N$ with a sufficiently
	large $u_1/u_2$ such that the following holds.
	First we consider the formal monomials
	$M_i=x_1^{r_{i,1}}\ldots x_{m-1}^{r_{i,m-1}}$
	with rational exponents for $1\leq i\leq \ell$.
	Then after the change of variables into the $y_j$'s,
	each $M_i$ becomes 
	a formal monomial in the $y_j$'s denoted by
	$y_1^{e_{i,1}}\ldots y_{m-1}^{e_{i,m-1}}$
	for $1\leq i\leq \ell$. With our choice of the $u_1,\ldots,u_{m-2}$, we have the following: 
	for $1\leq i, j\leq \ell$, if
	$$(r_{i,1},\ldots,r_{i,m-1}) \geq
	(r_{j,1},\ldots,r_{j,m-1})$$
	with respect to the lexicographic ordering
	on $\Q^{m-1}$ induced by the usual ordering $\geq$ on $\Q$
	then
	$$e_{i,1}\geq e_{j,1},\ldots,e_{i,m-1}\geq e_{j,m-1}.$$
	The power series obtained from $f$ after the change of variables
	into the $y_j$'s satisfies property $\cP$
	and its coefficients belong to a finite set if and
	only if the 
	$a_{\bfn}$'s do so.
	
	Therefore after a change of variables of the above form
	if necessary, we may assume that 
	for $1\leq i\neq j\leq \ell$ either
	$r_{i,k}\geq r_{j,k}$ for every $k\in\{1,\ldots,m-1\}$
	or $r_{i,k}\leq r_{j,k}$
	for every $k\in\{1,\ldots,m-1\}$.
	Moreover, by considering
	$f(x_1^L,\ldots,x_{m-1}^L,x_m)$
	where $L:=\lcm(n_{1,m},\ldots,n_{\ell,m})$, from now on
	we may assume that each $r_{i,k}\in \N_0$
	for every $1\leq i\leq \ell$ and $1\leq k\leq m-1$.

	Let $R(x_m)=\prod_{i=1}^{\ell}(1-\zeta_i\bfx^{n_i})$
	regarded as a polynomial in $x_m$ with coefficients
	in $\Qbar[x_1,\ldots,x_{m-1}]$ and we have
	$D:=\deg(R)=n_{1,m}+\ldots+n_{\ell,m}$. Then we can write:
	$$f(x)=\frac{P(x_1,\ldots,x_m)}{Q(x_1,\ldots,x_{m-1})\prod_{i=1}^{\ell}(1-\zeta_i\bfx^{\bfn_i})}=\sum_{N\geq 0} g_N(x_1,\ldots,x_{m-1})x_m^N.$$
	Each $g_N$ is a power series in the variables
	$x_1,\ldots,x_{m-1}$ and satisfies property $\cP$,
	hence each $g_N$ is a rational function and its denominator has the special form given in part (b). We also have that the coefficients of each $g_N$ belong to a finite set by the induction hypothesis. Such a finite set depends a priori on  $N$
	and our goal is to prove that there is a common finite set containing the coefficients
	of all the $g_N$'s.
	
	Observe that
	there is $N_1$ such that the sequence
	$(g_N)_{N\geq N_1}$ satisfies a linear recurrence relation
	whose characteristic polynomial
	is 
	$x_m^{D}R(1/x_m)$. 
	Since $1-\zeta_i\bfx^{\bfn_i}$ is irreducible
	for every $i$, we have that $\bfn_i$ is not a non-trivial
	integral multiple of a vector in $\N_0^m$. In particular, 
	$\bfn_i= \bfn_j$ if and only if
	$(r_{i,1},\ldots,r_{i,m-1})=(r_{j,1},\ldots,r_{j,m-1})$.
	Moreover, for $i\neq j$, since $1-\zeta_i\bfx^{\bfn_i}$
	and $1-\zeta_j\bfx^{\bfn_j}$ are distinct, if $\bfn_i=\bfn_j$ then we
	obviously have
	that $n_{i,m}=n_{j,m}$ and $\zeta_i\neq \zeta_j$.
	Therefore we have 
	exactly $D$ distinct characteristic roots
	denoted $\gamma_1,\ldots,\gamma_D$
	each of which has
	the form:
	$$\zeta_i^{1/n_{i,m}}M_i=\zeta_i^{1/n_{i,m}}x_1^{r_{i,1}}\ldots x_{m-1}^{r_{i,m-1}}$$
	for $1\leq i\leq \ell$ and each
	$\zeta_i^{1/n_{i,m}}$ denotes one of the $n_{i,m}$-th
	roots of $\zeta_i$. The difference of
	two different characteristic roots is either a 
	constant multiple of some $M_i$
	or has the form:
	\begin{equation}\label{eq:difference of 2 roots}
	\zeta_i^{1/n_{i,m}}x_1^{r_{i,1}}\ldots x_{m-1}^{r_{i,m-1}}-\zeta_j^{1/n_{j,m}}x_1^{r_{j,1}}\ldots x_{m-1}^{r_{j,m-1}}
	\end{equation}
	with $(r_{i,1},\ldots,r_{i,m-1})\neq (r_{j,1},\ldots,r_{j,m-1})$. Since we have that either 
	$r_{i,k}\geq r_{j,k}$ for every $k$ or
	$r_{i,k}\leq r_{j,k}$ for every $k$, the form
	\eqref{eq:difference of 2 roots}
	has the form
	\begin{equation}\label{eq:difference of 2 roots 2nd form}
	\xi_1 P_1(1-\xi_2 P_2)
	\end{equation}
	where $\xi_1$ and $\xi_2$ are roots of unity,
	$P_1$ and $P_2$ are monomials in $x_1,\ldots,x_{m-1}$.
	
	By the theory of linear recurrence sequences, we 
	have:
	$$g_N=\sum_{i=1}^D s_i\gamma_i^N$$
	for $N\geq N_1$ where the 
	$s_i$'s are the unique solution of the
	system of linear equations:
	$$s_1\gamma_1^{N_1+j}+\ldots+s_D\gamma_D^{N_1+j}=g_{N_1+j}\ \text{for $j\in\{0,\ldots,D-1\}$.}$$
	The determinant of the matrix 
	$(\gamma_i^{N_1+j})_{1\leq i\leq D,0\leq j\leq D-1}$
	is in $\Qbar[x_1,\ldots,x_{m-1}]$
	 and is equal to the product of 
	 a root of unity, a monomial,
	and polynomials of the form
	$(1-\zeta M)$
	where $\zeta$ is a root of unity and $M$ is
	a monomial in $x_1,\ldots,x_{m-1}$. Hence by Cramer's rule and the properties
	of the $g_{N_1+j}$ for $0\leq j\leq D-1$, each
	$s_i$ is a rational function whose denominator
	is the product of a monomial and 
	polynomials of the form $(1-\zeta M)$ as above.
	Replacing $f$ by its product with an appropriate monomial
	in $x_1,\ldots,x_{m-1}$ to cancel out the monomials in the denominators
	of the $s_i$'s, we may assume that 
	the denominator of each $s_i$ is the product
	of polynomials of the form $1-\zeta M$.
	
	Let $\ell'\leq \ell$ denote the number of distinct
	tuples among the tuples
	$(r_{i,1},\ldots,r_{i,m-1})$ for $1\leq i\leq \ell$.
	By relabelling those tuples, we may assume that 
	$(r_{i,1},\ldots,r_{i,m-1})$
	for $1\leq i\leq \ell'$ are
	all the distinct tuples
	and $r_{i,k}\leq r_{j,k}$ for every $k\in\{1,\ldots,m-1\}$
	and $i\leq j$.
	 Let 
	$N_2$ be the $\lcm$ of the orders of the roots
	of unity $\zeta_i^{1/n_{i,m}}$ for $1\leq i\leq \ell$. 
	For $0\leq \tau\leq N_2-1$, we restrict to
	the arithmetic progression $\{NN_2+\tau:\ N\geq 0\}$
	and get:
	$$g_{NN_2+\tau}=t_{\tau,1}M_1^{NN_2+\tau}+\ldots+t_{\tau,\ell'}M_{\ell'}^{NN_2+\tau}$$
	for all $N$ such that $NN_2+\tau\geq N_1$ where
	each $t_{\tau,k}$ is a rational function in $x_1,\ldots,x_{m-1}$ whose denominator 
	is the product of polynomials of the form
	$1-\zeta M$ as above. 
	Since the denominator of each $t_{\tau,k}$ has the mentioned
	form, it can be expressed as a power series
	in $x_1,\ldots,x_{m-1}$.
	
	Fix a $\tau\in\{0,\ldots,N_2-1\}$. Let $(e_1,\ldots,e_{m-1})\in \N_0^{m-1}$
	and let $c$ be the coefficient of $x_1^{e_1}\ldots x_{m-1}^{e_{m-1}}$
	in $t_{\tau,1}$. We choose $N$ so that $NN_2+\tau\geq N_1$ and 
	$$e_1+\ldots+e_{m-1}+(NN_2+\tau)\deg(M_1) < (NN_2+\tau)\deg(M_2).$$
	More specifically, let
	$$N=\max\left(\left\lceil \frac{e_1+\ldots+e_{m-1}}{N_2}\right\rceil,\left\lceil \frac{N_1}{N_2}\right\rceil\right)+1$$
	and we have that $c$ is the coefficient of
	$x_1^{e_1}\ldots x_{m-1}^{e_{m-1}} (x_1^{r_{1,1}}\ldots x_{m-1}^{r_{1,m-1}})^{NN_2+\tau}$
	in $g_{NN_2+\tau}$ which is also the coefficient of
	$x_1^{e_1}\ldots x_{m-1}^{e_{m-1}} (x_1^{r_{1,1}}\ldots x_{m-1}^{r_{1,m-1}})^{NN_2+\tau}x_m^{NN_2+\tau}$ in $f$.
	Since $f$ satisfies $\cP$, this implies that
	$h(c)=o(\log(e_1+\ldots+e_{m-1}))$, hence $t_{\tau,1}$ satisfies
	property $\cP$ as well.
	Having proved that $t_{\tau,1},\ldots,t_{\tau,i}$ satisfy
	property $\cP$, we use the equation
	$$g_{NN_2+\tau}-t_{\tau,1}M_1^{NN_2+\tau}-\ldots-t_{\tau,i}M_i^{NN_2+\tau}=t_{\tau,i+1}M_{i+1}^{NN_2+\tau}+\ldots+t_{\tau,\ell'}M_{\ell'}^{NN_2+\tau}$$
	and similar arguments to conclude that $t_{\tau,i+1}$ satisfies
	property $\cP$. 
	
	In conclusion, we have that $t_{\tau,i}$ satisfies $\cP$ for
	every $\tau\in\{0,\ldots,N_2-1\}$
	and $i\in\{1,\ldots,\ell'\}$. By the induction hypothesis, the coefficients
	of all those $t_{\tau,i}$'s
	belong to a finite set. Hence there is a finite set containing the coefficients
	of $g_{N}$ for $N\geq N_1$. Since the coefficients of each $g_N$ for $N<N_1$
	also contains in a finite set, we finish the proof.
	
	\subsection{Proof of Corollary~\ref{cor:B-C structure}}	
	We prove Corollary~\ref{cor:B-C structure} using
	standard specialization arguments. This gives 
	another proof of Theorem~\ref{thm:new B-C} in addition
	to the combinatorial method of Bell-Chen.

	Let $f(\bfx)\in K[[\bfx]]$ be D-finite and assume that
	the coefficients of $f$ belong to a finite set. By 
	Lemma~\ref{lem:Lipshitz}, we may assume $K=\bar{K}$ and
	for each $i=1,\ldots,m$, $f$ satisfies a linear partial
	differential equation as in the statement of this lemma.
	Let $R$ be the $\Qbar$-subalgebra of $K$ 
	generated by the coefficients of $f$ and the $A_{i,j}$'s
	and let $V$ be the affine algebraic variety with coordinate
	ring $R$. For every point $\zeta\in V(\Qbar)$, let
	$A_{i,j,\zeta}$ and $f_{\zeta}$ denote
	the corresponding specialization in $\Qbar[[\bfx]]$.
    We will consider $\zeta$ outside the proper Zariski closed
    subset defined by $A_{1,d_1}\ldots A_{m,d_m}=0$ so
    that the specializations of the given differential equations
    remain non-trivial.   
	
	By Noether normalization, there exist $y_1,\ldots,y_s\in R$
	algebraically independent over $\Qbar$ such that
	$R$ is finite over $\Qbar[y_1,\ldots,y_s]$ and this gives
	a finite surjective morphism
	$\pi:\ V\rightarrow \bA^s$. The set of points
	$(\alpha_1,\ldots,\alpha_s)\in\Qbar^s$ 
	where each $\alpha_i$ is a root of unity
	is Zariski dense in $\bA^s$. Each of the coefficients
	of $f$ and the $A_{i,j}$'s is a zero of a monic polynomial 
	with coefficients in
	$\Qbar[y_1,\ldots,y_s]$. Hence there is a positive constant $M$ 
	depending only on $R$ and a Zariski dense set 
	of $\zeta\in V(\Qbar)$ such that the following holds. For each
	$i=1,\ldots,m$, $f_{\zeta}$ satisfies the equation:
	$$\left(A_{i,d_i,\zeta}\left(\frac{\partial}{\partial x_i}\right)^{d_i}+\ldots+A_{i,1,\zeta}\frac{\partial}{\partial x_i}+A_{i,0,\zeta}(\bfx)\right)f_{\zeta}=0$$
	with $A_{i,d_i,\zeta}\neq 0$ and the heights of 
	the coefficients of $f_{\zeta}$ and the $A_{i,j,\zeta}$ 
	are bounded above by $M$. By Theorem~\ref{thm:main 1 effective}, 
	$A_{1,d_1,\zeta}\ldots A_{m,d_m,\zeta}f_{\zeta}$ is a polynomial
	and its total degree is bounded independently of
	$\zeta$. Since this holds for every $\zeta$ in a Zariski
	dense subset of $V(\Qbar)$, we have that
	$A_{1,d_1}\ldots A_{m,d_m}f$ is a polynomial and 
	this finishes the proof that $f$ is rational.
	
	Similarly, by Theorem~\ref{thm:new main 1}, we have that 
	for a Zariski dense set of points $\zeta\in V(\Qbar)$, 
	the denominator of $f_{\zeta}$ has
	the special form specified in
	part (b) of Theorem~\ref{thm:new main 1}. Therefore
	the denominator of
	$f$ has such a special form as well.

	\section{Proof of Theorem~\ref{thm:new main 2}}
	Let $d$, $k$, $\alpha_1,\ldots,\alpha_k$, $K$, $(a_n)_{n\geq 0}$
	be as in the statement of Theorem~\ref{thm:new main 2}.
	Assume that $f(x)=\displaystyle\sum_{n\geq 0}a_nx^n$
	is D-finite. By Lemma~\ref{lem:Lipshitz}, we have that $f$ 
	satisfies the equation:
	\begin{equation}\label{eq:sec 4 1}
	\left(P_D(x)\left(\frac{\partial}{\partial x}\right)^{D}+\ldots+P_{1}(x)\frac{\partial}{\partial x}+P_{0}(x)\right)f(x)=0
	\end{equation}
	where $P_{j}(x)\in \Qbar[x]$ for every
	$0\leq j\leq D$ and $P_{D}(x)\neq 0$. 
	If $D=0$, there is nothing to prove, so we may assume
	$D>0$. For $0\leq j\leq D$,
	let $S_j$ be the support of $P_j(x)$, let
	$B_j(x)\in\Z[x]$ be as in Section~\ref{subsec:Proof a}, and write
	$P_j(x)=\displaystyle\sum_{n\in S_j}p_{j,n}x^n$. As in the
	proof of Proposition~\ref{prop:1 differential eq},
	for every $r\in \N_0$, the coefficient of
	$x^r$ in the left-hand side of \eqref{eq:sec 4 1}
	is
	\begin{equation}\label{eq:sec 4 2}
	\sum_{j=0}^D\sum_{n\in S_j}p_{j,n}B_j(r+j-n)a_{r+j-n}=0
	\end{equation}
	with the convention that $a_n=0$ if $n\in\Z\setminus\N$.
	We now assume that $r$ is sufficiently large so
	that $r+j-n\geq 0$ for every $j\in\{0,\ldots,d\}$ and 
	$n\in S_j$. Then we apply the given formula
	for $a_{r+j-n}$ to \eqref{eq:sec 4 2} to obtain:
	\begin{equation}\label{eq:sec 4 3}
	\sum_{j=0}^D\sum_{n\in S_j}p_{j,n}B_j(r+j-n)\sum_{s=1}^k\sum_{t=0}^dc_{r+j-n,s,t}(r+j-n)^t\alpha_s^{r+j-n}=0.
	\end{equation}
	This equation can be written as
	$\displaystyle \sum_{s=1}^k \beta_{r,s}\alpha_s^r=0$
	where
	\begin{equation}\label{eq:sec 4 beta_{r,s}}
	\beta_{r,s}:=\sum_{j=0}^D\sum_{n\in S_j}\sum_{t=0}^d p_{j,n}B_j(r+j-n)c_{r+j-n,s,t}\frac{(r+j-n)^t}{\alpha_s^{n-j}}.
	\end{equation}
	By Proposition~\ref{prop:height properties}
	and the given properties of the $c_{r+j-n,s,t}$'s, we have:
	\begin{equation}\label{eq:sec 4 lim h}
	\lim_{r\to\infty}\frac{h(\beta_{r,s})}{(\log r)^{D+d+1}}=0
	\end{equation}
	for every $s=1,\ldots,k$. Consider the equivalence relation
	$\sim$ on $\{1,\ldots,k\}$ defined by
	$i\sim j$ if and only if
	$\alpha_i/\alpha_j$ is a root of unity. Assume there are
	$\gamma$ equivalence classes
	and let $s_1,\ldots,s_{\gamma}$ be the representatives. The equation $\displaystyle\sum_{s=1}^k\beta_{r,s}\alpha_s^r=0$
	can be rewritten as:
	\begin{equation}\label{eq:sec 4 s_gamma}
	\sum_{\ell=1}^{\gamma} \left(\sum_{i\sim s_{\ell}} 
	\beta_{r,i}\frac{\alpha_{i}^r}{\alpha_{s_{\ell}}^r}\right)\alpha_{s_{\ell}}^r=0.
	\end{equation}
	
	Note that the tuple $(\alpha_{s_{\ell}})_{\ell}$ is 
	non-degenerate and the height of each coefficient
	$\displaystyle\sum_{i\sim s_{\ell}} 
	\beta_{r,i}\frac{\alpha_{i}^r}{\alpha_{s_{\ell}}^r}$
	is $o((\log r)^{D+d+1})$. 
	By Proposition~\ref{prop:general SML}, we must have that
	\begin{equation}\label{eq:sec 4 13}
	\sum_{i\sim s_{\ell}} 
	\beta_{r,i}\frac{\alpha_{i}^r}{\alpha_{s_{\ell}}^r}=0
	\end{equation}
	for all $\ell=1,\ldots,\gamma$
	for all sufficiently large $r$. We now apply the same trick
	as in the proof of Proposition~\ref{prop:1 differential eq},
	each $\beta_{r,i}$ is a linear combination of
	$1,r,\ldots,r^{d+D}$  
	in which the height of each coefficient is 
	$o(\log r)$. Arguing as in the proof of Proposition~\ref{prop:1 differential eq}, \eqref{eq:sec 4 13} implies:
	\begin{equation}\label{eq:sec 4 14}
	\sum_{i\sim s_{\ell}}\sum_{n\in S_D}p_{D,n}\frac{c_{r+D-n,i,d}}{\alpha_i^{n-D}}\frac{\alpha_i^r}{\alpha_{s_{\ell}}^r}=0
	\end{equation}
	for all $\ell=1,\ldots,\gamma$ for all sufficiently large $r$.
	By multiplying both sides of \eqref{eq:sec 4 14} by
	$\alpha_{s_{\ell}}^r$ and summing over all 
	$\ell=1,\ldots,\gamma$,
	we obtain:
	\begin{equation}\label{eq:sec 4 15}
	\sum_{s=1}^k \sum_{n\in S_D} p_{D,n}c_{r+D-n,s,d}\alpha_s^{r+D-n}=0
	\end{equation}
	for all sufficiently large $r$. Put 
	$g(x)=\displaystyle\sum_{n\geq 0}(\sum_{s=1}^k c_{n,s,d}\alpha_s^n)x^n$ and observe that the left-hand side of 
	\eqref{eq:sec 4 15} is exactly the coefficient of
	$x^{r+D}$ in $P_Dg$. Therefore $g$ is a rational function. 
	Consider the operator $\Qbar[[x]]\rightarrow\Qbar[[x]]$
	given by 
	$$F\mapsto x\frac{\partial F}{\partial x}.$$
	By applying this operator to $g$ for $d$ many times, we can show
	that 
	$$\tilde{g}(x)=\sum_{n\geq 0}(\sum_{s=1}^k c_{n,s,d}n^d\alpha_s^n)x^n$$
	is a rational function. This yields two things. First,
	Theorem~\ref{thm:new main 2} holds when $d=0$. Second,
	the power series
	$$f(x)-\tilde{g}(x)=\sum_{n\geq 0}(\sum_{s=1}^k\sum_{t=0}^{d-1} c_{n,s,t}n^t\alpha_s^n)x^n$$
	is D-finite so that we can finish the proof by using induction.
	
	\begin{remark}
	From the above proof, we have that if $f$ is D-finite then
	for each $t\in\{0,\ldots,d\}$, the 
	power series
	$\displaystyle\sum_{n\geq 0} (\sum_{s=1}^kc_{n,s,t}\alpha_s^n)x^n$
	is a rational function.
	\end{remark}
	
	\section{Proof of Theorem~\ref{thm:new main 3}}
	Since $h(a_n)\leq h(a_0,\ldots,a_n)$, it remains
	to show that:
	$$\limsup_{n\to\infty}\frac{h(a_0,\ldots,a_n)}{n\log n}<\infty.$$
	Let $K$ be a number field containing the coefficients
	of $f$. As before, we have that
	that the coefficients $(a_n)$ eventually satisfy a linear recurrence relation with polynomial
	coefficients. In other words, there exist 
	$M\in\N_0$ and polynomials
	$R_0(t),\ldots,R_M(t)\in \Qbar[t]$
	with $R_M\neq 0$
	such that
	\begin{equation}\label{eq:P-recursive}
	R_M(n)a_{n+M}+\ldots+R_0(n)a_n=0
	\end{equation}
	for all sufficiently large $n$. 
	
	Let $v\in M_K$. If $v$ is non-archimedean, we have: 
	$$\vert a_{n+M}\vert_v\leq \max_{0\leq i\leq M-1} \left\vert\frac{R_{i}(n)}{R_M(n)}a_{n+i}\right\vert_v$$
	which implies:
	$$\max_{0\leq i\leq n+M}\log^{+}\vert a_{i}\vert_v
	\leq \max_{0\leq i\leq n+M-1}\log^{+}\vert a_i\vert_v
	+\max_{0\leq i\leq M-1}\log^{+}\left\vert \frac{R_i(n)}{R_M(n)}\right\vert_v.
	$$
	If $v$ is archimedean, we have:
	$$\vert a_{n+M}\vert_v\leq \vert M\vert_v\max_{0\leq i\leq M-1} \left\vert\frac{R_{i}(n)}{R_M(n)}a_{n+i}\right\vert_v$$
	which implies:
	$$\max_{0\leq i\leq n+M}\log^{+}\vert a_{i}\vert_v
	\leq \log\vert M\vert_v + \max_{0\leq i\leq n+M-1}\log^{+}\vert a_i\vert_v
	+\max_{0\leq i\leq M-1}\log^{+}\left\vert \frac{R_i(n)}{R_M(n)}\right\vert_v.
	$$
	
	Summing over all $v$, we have:
	$$h(a_0,\ldots,a_{n+M})-h(a_0,\ldots,a_{n+M-1})=O(\log n)$$
	and this yields the desired result.

	\bibliographystyle{amsalpha}
	\bibliography{Dfinite_bib} 	

\end{document}